\documentclass{amsart}
\pdfoutput=1
\usepackage{amsmath,amssymb,amsthm}\usepackage[mathscr]{eucal}
\usepackage[dvipsnames]{xcolor}
\usepackage{amscd}\usepackage{setspace}\usepackage{epsfig}\usepackage{color}
\usepackage{multicol}\usepackage{dsfont}\usepackage{marvosym}
\usepackage{wasysym}\usepackage{mathdots}\usepackage{bbold} 
\usepackage{tikz}
\usepackage{tensor}
\usetikzlibrary{matrix}

\newtheorem*{maintheorem}{Main theorem}
\newtheorem*{ctlemma}{Crossing tube lemma}
\newtheorem*{bigonmove}{Bigon move}
\newtheorem{theorem}{Theorem}[subsection]
\newtheorem{lemma}[theorem]{Lemma}

\newtheorem{prop}[theorem]{Proposition}

\theoremstyle{definition}

\begin{document}

\title{Alternating links have representativity 2} 

\author[T. Kindred]{Thomas Kindred}
\address{Department of Mathematics, University of Iowa,
Iowa City, IA 52242-1419, USA}
\email{thomas-kindred@uiowa.edu}

\date{\today}

\maketitle

\begin{abstract}  
We prove that if $L$ is a non-trivial alternating link embedded (without crossings) in a closed surface $F\subset S^3$, then $F$ has a compressing disk whose boundary intersects $L$ in no more than two points.  Moreover, whenever the surface is incompressible and $\partial$-incompressible in the link exterior, it can be isotoped to have a standard tube at some crossing of any reduced alternating diagram.  
\end{abstract}

\section{Introduction}\label{S:1}

Following Ozawa \cite{ozawa1,ozawa2}, the \textbf{representativity} $r(L)$ of a link $L\subset S^3$ is: 
\[r(L)=\max_{F\in\mathcal{F}_L}\min_{X\in\mathcal{X}_F}|\partial X\cap L|,\]
where $\mathcal{F}_L$ is the set of positive genus, closed surfaces $F\subset S^3$ containing $L$, and a ``closed surface'' is compact and connected without boundary;  $\mathcal{X}_F$ is the set of compressing disks for $F$ in $S^3$; and $|\partial X\cap L|$ is the number of connected components (i.e. points) of $\partial X\cap L$. This notion extends the earlier notion of representativity from graph theory. In 2011 \cite{pardon}, Pardon applied representativity (although he did not use this term) to answer a question posed by Gromov \cite{gromov} in 1983 regarding knot distortion.  The distortion of an embedded circle $\gamma$ in $\mathbb{R}^3$ is defined to be:
\[\delta(\gamma)=\sup_{p,q\in\gamma}\frac{d_\gamma(p,q)}{d_{\mathbb{R}^3}(p,q)},\]
where $d_\gamma$ is arclength along $\gamma$, and $d_{\mathbb{R}^3}$ is Euclidean distance in $\mathbb{R}^3$.  Gromov asked whether there exists a uniform upper bound on distortion for all isotopy classes of knots, or at least for  torus knots. Specifically, Gromov asked, does every isotopy class of knots have a representative with distortion less than, say, 100?  To answer this question, Pardon showed that every knot isotopy class $K$ satisfies:
\[\delta(K):=\min_{ \text{representatives }\gamma \text{ of }K}\delta(\gamma)\geq\frac{1}{160} r(K),\]
where $r(K)$ denotes representativity. In particular, since the representativity of any $(p,q)$ torus knot is $r(T_{p,q})= \text{min}\{p,q\}$ (more to the point and easier to check is $r(T_{p,q})\geq\min\{p,q\}$), so that $\delta(T_{p,q})\to\infty$ as $p,q\to\infty$, Pardon was able to answer Gromov's question in the negative. Current work of Blair, Campisi, Taylor and Tomova \cite{bt} provides a lower bound for $\delta(K)$ in terms of distance and bridge numbers.  Our main result implies that their lower bound improves Pardon's lower bound in the case of alternating knots with sufficiently large distance:
\begin{maintheorem}
Every non-trivial, non-split alternating link $L$ has representativity $r(L)=2$.
\end{maintheorem}
As another application, the main theorem gives a new proof of the fact that the only alternating torus links are the 2-braids $T_{2,q}$, again since $r(T_{p,q})=\min\{p,q\}$. As $T_{3,4}$ is almost-alternating, the main theorem does not extend to this class of links. We return to this example in \textsection\ref{S:5},  suggesting a possible approach to the characterization of almost-alternating links with representativity $\geq 3$.

To prove the main theorem, we employ the crossing ball structures introduced by William Menasco \cite{men}. Roughly, the idea is to insert a ball $C_t$ at each crossing of a given diagram $D$ on $S^2$, to perturb $L$ to lie on $(S^2\setminus C)\cup\partial\nu C$, where $C=\bigsqcup C_t$, and then to isotope a given closed surface $F\supset L$ (fixing $L$ and the crossing ball structure $S^2\cup C$) so as to minimize its intersections with $C$ and $S^2$ away from $L$. 
We show that whenever $F$ is essential (incompressible and $\partial$-incompressible in the link exterior $S^3\setminus \text{int}(\nu L)$), there exists an isotopy of $F$ which produces a standard tube near some crossing (cf. Figure \ref{Fi:Smoothing}).

\begin{ctlemma}
Given a non-trivial, connected, reduced alternating diagram of a link $L$ and a closed, essential surface $F\supset L$, 
there exists an isotopy of $F$ after which some crossing has a standard \textbf{tube}.
\end{ctlemma}

This lemma not only provides a compressing disk for $F$ whose boundary intersects $L$ in two points; it also provides a possible inductive move, in the tradition of \cite{gab,ak}, albeit one still awaiting application.

 \textbf{Thank you} to Seungwon Kim for sharing this problem during a visit to Iowa; to Colin Adams for introducing the author to Menasco's techniques (in particular to ``bigon moves'');  to Maggy Tomova and Ryan Blair for helpful discussions around early and final versions of this paper, respectively; and to Charlie Frohman for patient and inspired coaching.


\section{Initial setup}\label{S:2}

\subsection{Link diagrams and crossing balls}\label{S:21}
A {link diagram} $D\subset S^2$ is the image of an immersion of one or more circles in $S^2\subset S^3$ in which all self-intersections are double-points at which the two intersecting arcs are transverse in $S^2$ and are labeled with opposite normal directions relative to $S^2$. Thus, a link diagram $D$ can also be seen as a smoothly embedded 4-valent graph in $S^2$ with over-under information at each vertex.  An embedding of the underlying link $L$ can be obtained by perturbing the two arcs of $D$ near each crossing point in the indicated normal directions.  

A link diagram $D$ is called \textbf{alternating} if, for each edge of $D$ (seen as a 4-valent graph), the crossing points at its two ends are labeled with {opposite} normal directions. A link  is called {alternating} if it has an alternating diagram. A link diagram $D$ is called \textbf{reduced} if it lacks \textit{nugatory} crossings, which means that every crossing point is incident to four distinct components of $S^2\setminus D$. 
A link $L$ is called \textbf{split} if $S^3\setminus L$ is reducible, i.e. if there is an embedded 2-sphere in $S^3\setminus L$ which does not bound a ball in $S^3\setminus L$. 


Let $D$ be a diagram of a link $L$ (in \textsection\ref{S:3}-\ref{S:4}, we will assume further that $D$ is non-trivial, connected, reduced, and alternating) with crossing points $c_t$, $t=1,\hdots, n$. Insert small, mutually disjoint (closed) \textbf{crossing balls} $C_t$, $t=1,\hdots, n$, centered at the respective crossing points $c_t$. Denote $C=\bigsqcup_{t=1}^nC_t$.  Construct an embedding of $L$ in $(S^2\setminus C)\cup\partial C$ by perturbing the two arcs in which $D$ intersects each crossing ball $C_t$ in their indicated normal directions from $S^2\cap C_t$ to $\partial C_t$, while fixing $D\cap S^2\setminus \text{int}(C)$.  

Call each resulting component of $L\cap S^2$ an  \textbf{edge} of $L$---note that $L\cap S^2=L\cap D=D\setminus \text{int}(C)$---and call each component of $L\cap\partial C$ an  \textbf{overpass} or an  \textbf{underpass} of $L$, according to which side of $S^2$ it lies on. Near each crossing, this looks like Figure \ref{Fi:crossingball} (center).

\begin{figure}
\begin{center}
\includegraphics[width=6.5in]{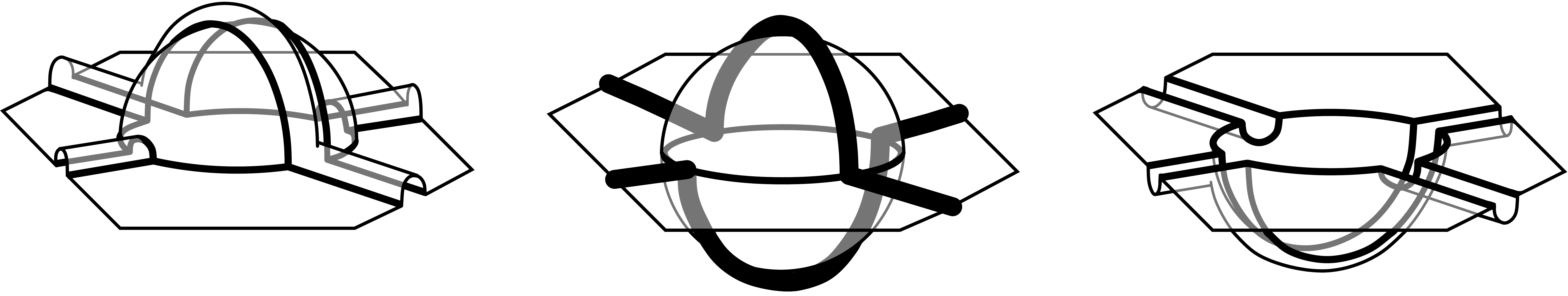}
\caption{A link near a crossing ball (center), with $S^+$ (left) and $S^-$ (right).}
\label{Fi:crossingball}
\end{center}
\end{figure}

\subsection{The regular neighborhood $\nu L$}\label{S:22}
 Let $\nu L\subset S^3$ be a {closed}  regular neighborhood of the link $L$, viewed as (the total space of) a disk-bundle $\pi:\nu L\to L$ for which the restrictions $\pi|_{\partial C}$, $\pi|_{S^2\setminus \text{int}(C)}$ are also bundle maps onto their images (each with fiber a closed interval). 
 Thus, each component of $\nu L$ intersects $(S^2\setminus C)\cup\partial C$ in an annular neighborhood of its core; and, for each point $p\in L\cap S^2\cap\partial C$, the boundary of the disk $\pi^{-1}(p)$ is a meridian on $\partial\nu L$ which consists of an arc of $\partial\nu L\cap\partial C\cap S^\pm$ glued to an arc of $\partial\nu L\cap S^\mp\setminus C$ at the two points of $\pi^{-1}(p)\cap\partial\nu L\cap S^+\cap S^-$.


Let $B^+$ and $B^-$ be the two components into which $S^2\cup C\cup\nu L$ cuts $S^3$; i.e., $B^\pm$ are the 
closures of the components of $S^3\setminus (S^2\cup C\cup\nu L)$. 
Let $S^+=\partial B^+$ and $S^-=\partial B^-$.  Near each crossing, $S^+$ and $S^-$ appear as in Figure \ref{Fi:crossingball} (left and right).
As a quick exercise, check that $S^+\cup S^-=(S^2\setminus (C\cup\nu L))\cup\partial (C\cup\nu L)$ and $S^+\cap S^-=S^2\setminus \text{int}(C\cup\nu L)$.    

Use this setup to extend the terminology of edges, overpasses, and underpasses from $L$ to $\partial\nu L$ as follows.
Call each component of $\pi^{-1}(L\cap\partial C)\cap\partial\nu L\cap S^+$ an  \textbf{overpass} of $\partial\nu L$,
call each component of $\pi^{-1}(L\cap\partial C)\cap\partial\nu L\cap S^-$ an  \textbf{underpass} of $\partial\nu L$, and 
call each component of $\pi^{-1}(L\cap S^2)\cap\partial\nu L$ an  \textbf{edge-pair}; each edge-pair is the union of two \textbf{edges} of $\partial\nu L$, one a component of $\pi^{-1}(L\cap S^2)\cap S^+$, the other a component of $\pi^{-1}(L\cap S^2)\cap S^-$.  
This terminology gives $\partial\nu L$ the following cell decomposition (cf. Figures \ref{Fi:crossingball}, \ref{Fi:EdgeBlank}).  
\begin{itemize}
\item
 The 0-cells are the points of $\partial\nu L\cap S^2\cap\partial C$
 , eight on the boundary of each crossing ball.
\item
   There are several types of (closures of) 1-cells:
\begin{itemize}
\item arcs of $\partial\nu L\cap S^+\cap S^-$,  two running along each edge-pair \textit{of $\partial\nu L$};
\item arcs of $\partial\nu L\cap\partial C\cap S^+$, two along each overpass of $\partial\nu L$;
\item arcs of $\partial\nu L\cap\partial C\cap S^-$,  two along each underpass of $\partial\nu L$; and
\item arcs of $\pi^{-1}(L\cap C\cap S^2)\cap S^\pm$, eight near each crossing ball $C_t$: 
\begin{itemize}
\item four in $\partial\nu L\cap\partial C_t$,
\item two in $S^+\cap\partial\nu L\setminus\partial C$, joining the overpass of $\partial\nu L$ at $C_t$ with edges of $\partial\nu L$, and
\item two in $S^-\cap\partial\nu L\setminus\partial C$, joining the underpass of $\partial\nu L$ at $C_t$ with edges of $\partial\nu L$.
\end{itemize}
\end{itemize}
\item The 
  2-cells' closures are overpasses, underpasses, and edges of $\partial\nu L$; plus components of $\partial\nu L\cap C$.
\end{itemize}
\begin{figure}
\begin{center}
\includegraphics[height=1.4in]{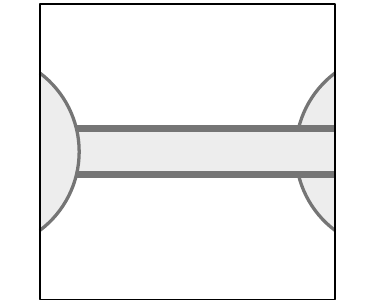}
\hspace{.25in}
\includegraphics[height=1.7in]{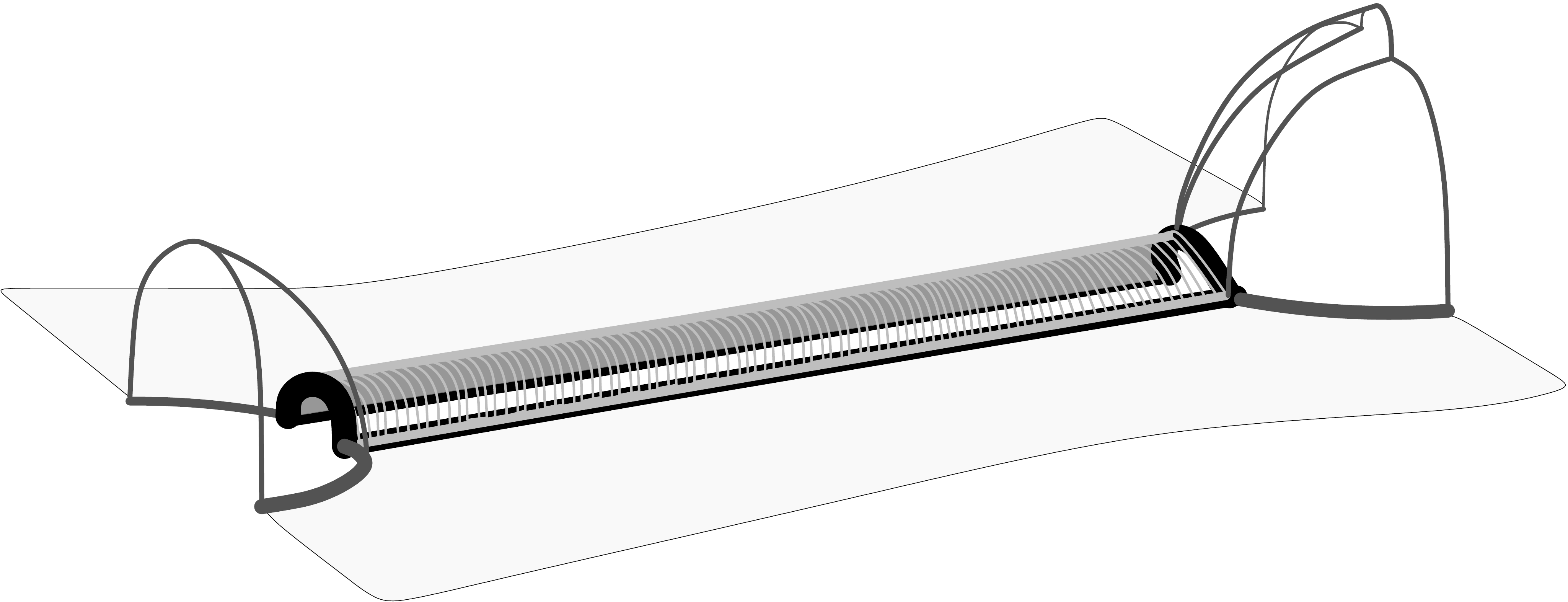}
\caption{Two views of an edge of $\partial\nu L\cap S^+$ from an alternating link diagram. }\label{Fi:EdgeBlank}
\end{center}
\end{figure}
\subsection{Essentiality and minimal complexity for the closed surface $F\supset L$}\label{S:23}
So far, the link $L\subset (S^2\setminus C)\cup\partial C$ follows a 
link diagram $D\subset S^2$; $\nu L$ is a closed regular neighborhood of $L$ in $S^3$, seen as (the total space of) a disk-bundle $\pi:\nu L\to L$; and $B^\pm$ are the components into which $S^2\cup C\cup\nu L$ cuts $S^3$, with $S^\pm=\partial B^\pm$. Now also
 let $F$ be any closed surface in $S^3$ of positive genus that contains $L$. (Recall that a ``closed surface'' is assumed to be compact and connected without boundary). 
Fixing $L\subset F$, $S^2$, and $C$, isotope $F$ so that:
{\begin{itemize}
\item $F$ is transverse to $S^+$ and $S^-$;
\item the restriction $\pi|_F$ is a bundle map (so that $F\cap\nu L$ is a regular neighborhood of $L$ in $F$ and each component of $F\cap\partial\nu L$ is the image of a section of $\pi:\nu L\to L$); and
\item $F\cap C\cap\partial\nu L=\varnothing$ (so that $F$ intersects each overpass and underpass of $\partial\nu L$ in precisely two arcs).
\end{itemize}}
Perform such isotopy so as to \textbf{minimize lexicographically} the numbers of components of $F\cap\partial C\setminus\nu L$
 and of $F\cap S^+\cap S^-$, i.e. to minimize the \textbf{complexity} $\left(|F\cap\partial C\setminus\nu L|,|F\cap S^+\cap S^-|\right)$, where bars count connected components. Note that the first and last conditions above ensure that $F\cap\partial C\setminus\nu L$ will consist only of simple closed curves, since the endpoints of any arc of $F\cap\partial C\setminus\nu L$ would lie on $\partial C\cap\partial\nu L$.

The surface $F$ is said to be \textbf{compressible} in the link exterior $E=S^3\setminus \text{int}(\nu L)$ if there is a disk $X\subset E$ with $ X\cap F=\partial X$ a simple closed curve (a ``circle'') that does not bound a disk in $F$. The surface $F$ is compressible in $E$ if and only if $r(F,L)=0$. In particular, every split link $L$ has representativity $r(L)=0$.

The surface $F$ is said to be \textbf{$\partial$-compressible} in the link exterior $E$ if there is an arc $\alpha$ on $\partial\nu L$ 
which is parallel to $F$ through $E$---say, through a disk $X$, with $\beta$ denoting the arc $\partial X\cap F$---but not through $\partial\nu L$.   In this case, $\alpha$ is also parallel through a bigon $Y\subset\nu L$ to an arc $\alpha'\subset F\cap\nu L$ which intersects $L$ in one point. Gluing $X$ and $Y$ along $\alpha$ produces a compressing disk $Z$ for $F$ in $S^3$ whose boundary $\partial Z=\alpha'\cup\beta$ intersects $L$ in one point.  Thus, if $F$ is $\partial$-compressible in $E$, then $r(F,L)\leq 1$; conversely as well.  
 
 In particular, the trivial {knot} $L$ has representativity $r(L)=1$, since any positive genus, closed surface containing $L$ must be compressible or $\partial$-compressible.  A closed surface $F\supset L$ is called \textbf{essential} if it is neither compressible nor $\partial$-compressible in the link exterior.  

Any non-trivial \textit{knot} $L$ is contained in a positive genus, closed surface $F$ with $r(F,L)=2$, namely the {interpolating surface} for any incompressible Seifert surface for $L$, or more generally for any algebraically essential spanning surface for $L$. (A spanning surface $V$ for $L$ is an embedded surface with boundary $\partial V=L$ in $S^3$; the interpolating surface for a spanning surface $V$ is the boundary of a regular neighborhood of $V$ in the link exterior; and $V$ is called algebraically essential if its interpolating surface is incompressible and $\partial$-incompressible in the link exterior.) Is this also true of non-split \textit{links} $L$? Is $r(L)\geq 2$ \textit{if} and only if $L$ is non-split and non-trivial? This is true for all non-split, non-trivial links with an algebraically essential, \textit{connected} spanning surface, using the interpolating surface.  Does every non-split link have such a span? 

At least in the alternating case, the answer is yes (in fact, \textit{all} spanning surfaces for non-split alternating links are connected \cite{ak}). Thus, an alternating link $L$ in a reduced alternating diagram $D$ obeys: 
\begin{align*}
r(L)=0&\iff~L \text{ is split }\iff~D \text{ is disconnected;}\\
r(L)=1&\iff~L \text{ is the unknot }\iff~D \text{ is the trivial diagram.}\\
\end{align*}
The main theorem states that whenever $L$ is alternating, non-trivial, and non-split, $r(L)=2$.
\subsection{A preliminary consequence of the initial setup}\label{S:24}
The three results in \textsection\ref{S:24} assume the following setup from \textsection\ref{S:21}-\ref{S:23}.
A link $L\subset (S^2\setminus C)\cup\partial C$ follows a non-trivial, connected link diagram $D\subset S^2$ with $n$ crossings, and $\nu L$ is a regular neighborhood of $L$, seen as (the total space) of a disk-bundle $\pi:\nu L\to L$.
Balls $B^\pm$ are the closures of the two components of $S^3\setminus (S^2\cup C\cup\nu L)$, with $S^\pm=\partial B^\pm$ and $S^+\cap S^-=S^2\setminus \text{int}(C\cup\nu L)$.
And a closed essential surface $F\supset L$ has been isotoped, subject to the conditions that $F$ is transverse to $S^+$ and $S^-$, $\pi|_F$ is a bundle map, and $F\cap C\cap\partial\nu L=\varnothing$, so as to minimize its complexity $\left(|F\cap\partial C\setminus\nu L,F\cap S^+\cap S^-|\right)$. This setup implies:
\begin{prop}\label{P:Disks}
All components of $F\setminus (S^+\cup S^-\cup\nu L)$ are disks.
\end{prop}
\begin{proof}
The $n+2$ components of $S^3\setminus (S^+\cup S^-\cup\nu L)$---namely $B^+$, $B^-$, and $C_t\setminus \text{int}(\nu L)$, $t=1,\hdots, n$---are all topological 3-balls, and their boundaries are the 2-spheres $S^+$, $S^-$, and $\partial (C_t\setminus\nu L)$. These spheres intersect $F$ transversally, hence in circles.  Each such circle bounds a disk in the corresponding ball. 

The incompressibility of $F$ implies that each of these circles also bounds a disk in $F$. Minimality then implies that this disk in $F$ must lie entirely in the appropriate ball, as claimed. Specifically, each circle of $F\cap S^\pm$ bounds a disk of $F\cap B^\pm$, and each circle of $F\cap (\partial C_t\setminus \text{int}(\nu L))$ bounds a disk of $F\cap C_t\setminus \text{int}(\nu L)$.
\end{proof}
%
%

%

\section{Technical conveniences}\label{S:broader}
Throughout \textsection\ref{S:broader}, maintain all setup from \textsection\ref{S:21}-\ref{S:23}, but replace the assumption that the complexity of $F$ is minimized with the assumption that all components of $F\setminus (S^+\cup S^-\cup\nu L)$ are disks.  Proposition \ref{P:Disks} implies that this setting is more general than the initial setup.  
 
 \subsection{Preliminary consequences in the broader setting -- arcs and balls.}
 \begin{prop}\label{P:Arcs}
All components of $F\cap\partial\nu L\cap S^\pm$, $F\cap\partial C\cap S^\pm$, and $F\cap S^+\cap S^-$ are arcs.
\end{prop}
\begin{proof}
Every component of $\nu L$ contains an overpass and an underpass, or else $L$ would be a split link, and $F$ (assumed to be connected) would be compressible, contrary to assumption.  Therefore, each component of $F\cap\partial\nu L$, (the image of) a section of $\pi:\nu L\to L$ by assumption, must intersect $S^+\cap S^-$.  Hence, no component of $F\cap\partial\nu L\cap S^\pm$ is a circle; instead, each 
 must be an arc.
 
All components of $\partial C\cap S^+$, $\partial C\cap S^-$, and $S^+\cap S^-$ are disks. If $F$ intersected one of these disks in a circle, $\gamma$, then all components of $F\setminus (S^+\cup S^-\cup\nu L)$ are disks, $\gamma$ would bound disks of $F$ in both components of  $F\setminus (S^+\cup S^-\cup\nu L)$ whose boundaries contain $\gamma$.  This contradicts the assumption that $F$ is connected.
\end{proof}
The conclusion of Proposition \ref{P:Arcs}, implies that the number of components of  $F\cap S^+\cap S^-$ equals half the number of points of $F\cap\partial (S^+\cap S^-)
$, which will be more convenient to count.  Also, the conclusions of Propositions \ref{P:Disks} and \ref{P:Arcs} together imply that $F\setminus \text{int}(\nu L)$ has the following cell decomposition: 
\begin{itemize}
\item The 0-cells are the points of $F\cap\partial (S^+\cap S^-)$.
\item There are three types of (closures of) 1-cells: arcs of $F\cap S^+\cap S^-$,  $F\cap\partial C\cap S^\pm$, and  $F\cap\partial\nu L\cap S^\pm$.
\item The (closures of the) 2-cells are the components of $F\cap B^+$, $F\cap B^-$, and $F\cap C\setminus\nu L$.
\end{itemize}
\begin{prop}\label{P:Balls}
 All components of $B^\pm\setminus F$ are 3-balls.
\end{prop}
\begin{proof}
Let $A$ be a component of $B^\pm\setminus F$, and let $X$ be a component of $\partial A$.  Since $F$ is connected, $X$ must contain (at least) one component $Y$ of $S^\pm\setminus F$. 
The boundary of $Y$ is a union of circles of $F\cap S^\pm$, each of which bounds a disk of $F\cap B^\pm$, by assumption.  The union of these disks and $Y$ is a closed surface and therefore equals $X$; computing euler characteristic reveals that $X$ is a 2-sphere in $B^\pm$ and therefore bounds a 3-ball in $B^\pm$.  This 3-ball must be $A$; otherwise, since each component of $\partial A$ is a 2-sphere which bounds a 3-ball in $B^\pm$, gluing each of these 3-balls to $A$ along $\partial A$ would yield a closed 3-manifold (compact and connected without boundary) contained in $B^\pm$.
\end{proof}

\subsection{Bigon moves}\label{S:bigon}
\begin{figure}
\begin{center}
\includegraphics[width=6.5in]{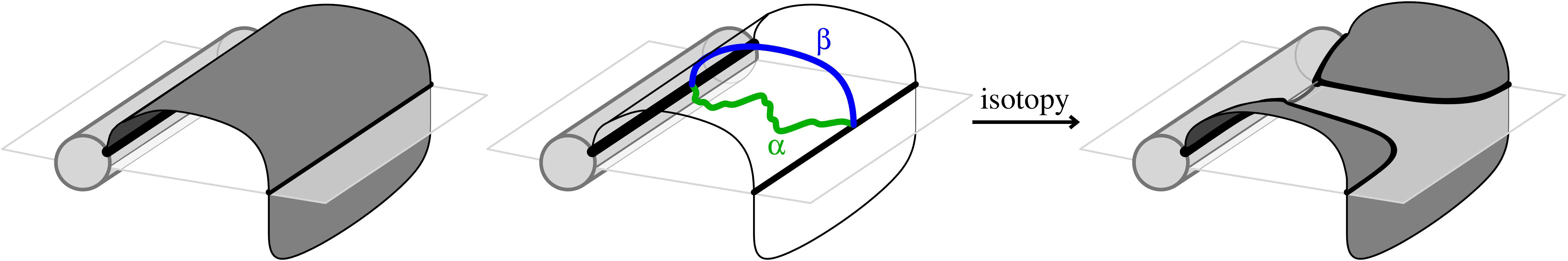}
\caption{Bigon moves will often prove useful, even when they increase the complexity of $F$.}
\label{Fi:bigon}
\end{center}
\end{figure}
\begin{figure}
\begin{center}
\includegraphics[width=6.5in]{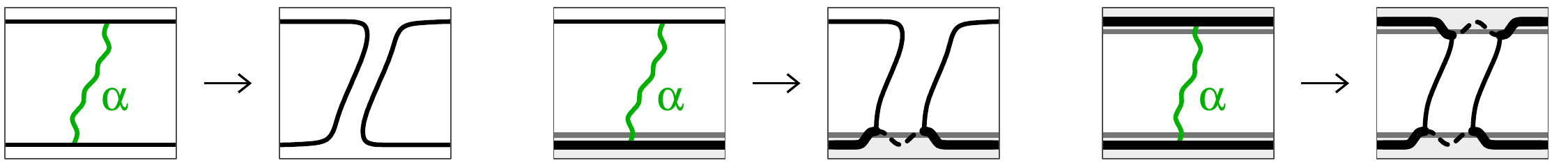}
\caption{A bigon move pushes an arc $\beta\subset F\cap B^\pm$ past a parallel arc $\alpha\subset S^\pm\setminus\pi^{-1}(C)$, provided $\alpha$ is not parallel to $F$ through $S^\pm\setminus C$ and 
 $|\alpha\cap S^+\cap S^-|=1$.}
\label{Fi:move}
\end{center}
\end{figure}
Several proofs in \textsection\ref{S:3}-\ref{S:4} will use an isotopy move which pushes an arc $\beta\subset F\cap B^\pm$ past a parallel arc $\alpha\subset S^\pm\setminus F$ through a disk $Z$ with $Z\cap F=\beta\subset\partial Z=\alpha\cup\beta$. One type of this ``bigon move'' is illustrated in Figure \ref{Fi:bigon}; all three types are diagrammed in Figure \ref{Fi:move}. More precisely, a bigon move follows an arc $\alpha\subset S^\pm$ that:
\begin{itemize}
\item intersects $F$ precisely on its endpoints, which lie on the same circle $\gamma$ of $F\cap S^\pm$;
\item is not parallel in $S^\pm\setminus C$ to $F$; 
\item is disjoint from $\pi^{-1}(C)$, i.e. from crossing balls and over/underpasses; and
\item intersects $S^+\cap S^-$ in exactly one component. 
\end{itemize}
Think of $\alpha$, which initially is not ``part of the diagram,''   as a marker which joins two points that lie on the same circle $\gamma$ of $F\cap S^\pm$, but ``not obviously'' so, locally.  
Because the circle 
$\gamma$ bounds a disk $Y\subset F\cap B^\pm$,  there is an arc $\beta\subset Y$ with the same endpoints as $\alpha$. The circle $\alpha\cup\beta$ lies on the boundary of some component of $B^\pm\setminus F$, a ball, 
and therefore bounds a disk $Z\subset B^\pm$ whose interior is disjoint from $F$.  The disk $Z$ is a bigon in the sense that $\partial Z=\alpha\cup\beta$.
\begin{bigonmove}\label{BigonMove}
Given an arc $\alpha\subset S^\pm$ that satisfies the four conditions above, it is possible to isotope $F$ near a parallel arc $\beta\subset F\cap B^\pm$ through a bigon $Z$ past $\alpha$.
\end{bigonmove}
All bigon moves take place away from the crossing balls, and thus preserve both the number of components of  $F\cap\partial C\setminus\nu L$ and the fact that each of these components bounds a disk of $F\cap C\setminus \nu L$. 
Not all bigon moves, however, preserve the fact that all components of $F\setminus (S^+\cup S^-\cup \nu L)$ are disks. At least:
\begin{lemma}\label{L:Move}
Performing a bigon move preserves the fact that all components of $F\setminus (S^+\cup S^-\cup \nu L)$ are disks whenever $\alpha\cap\partial\nu L\neq\varnothing$ and whenever the complexity of $F$ is minimized.
\end{lemma}
\begin{proof}
As noted, any bigon moves preserves the fact that all components of $F\cap C\setminus\nu L$ are disks.  
In order for a bigon move to upset the fact that all components of $F\cap B^+$ and $F\cap B^-$ are disks, a necessary condition is that the endpoints of $\alpha$ lie on the same circle of $F\cap S^\pm$ and also on the same circle of $F\cap S^\mp$. For this to be the case, both endpoints of $\alpha$ must lie in $S^+\cap S^-$, as must all of $\alpha$, since $\alpha\cap S^+\cap S^-$ is connected by assumption. Thus, the condition $\alpha\cap\partial\nu L\neq\varnothing$ in the statement of the Lemma suffices as claimed.

Suppose instead that, with the complexity of $F$ minimized, a bigon move follows an arc $\alpha\subset S^+\cap S^-$ whose endpoints lie both on the same circle $\gamma=\partial Y$ of $F\cap S^\pm$ \textit{and} on the same circle $\gamma'=\partial Y'$ of $F\cap S^\mp$. Letting $Y$, $Y'$ denote the disks of $F\cap B^+$, $F\cap B^-$ respectively bounded by $\gamma$, $\gamma'$, there are arcs $\beta\subset Y$, $\beta'\subset Y'$ with the same endpoints as $\alpha$. 
Further, $\alpha\cup\beta$, $\alpha\cup\beta'$ respectively bound disks $Z\subset B^\pm$, $Z'\subset B^\mp$ whose interiors are disjoint from $F$. 
Gluing $Z$ and $Z'$ along $\alpha$ produces a disk $Z\cup Z'$ with boundary $\beta\cup\beta'\subset F\setminus\nu L$ whose interior is disjoint from $F$; since $F$ is incompressible in the link exterior, $\beta\cup\beta'$ must bound a disk $X\subset F\setminus\text{int}(\nu L)$. 
Since $L$ is non-split, the 2-sphere $Z\cup Z'\cup X\subset S^3\setminus\nu L$ bounds a 3-ball $W\subset S^3\setminus\nu L$
, through which $X$ is parallel to $Z\cup Z'$. 

Since the complexity of $F$ is minimized, the disk $X$, like $Z\cup Z'$, must be disjoint from $C$ and must intersect $S^+\cap S^-$ in a single arc, $\delta$. From $X\cap C=\varnothing$, it follows that $\partial W\cap C=\varnothing$. This and the fact that $W\cap\nu L=\varnothing$ imply that $W$ is disjoint from $C\cup \nu L$, and in particular from $\partial(S^+\cap S^-)$. Therefore, contrary to assumption, $W$ intersects $S^+\cap S^-$ in a single disk through which $\alpha$ is parallel to the arc $\delta\subset F$:
\[
\pushQED{\qed}
\partial(W\cap S^+\cap S^-)=\left(\partial W\cap (S^+\cap S^-)\right)\cup\left(W\cap\partial(S^+\cap S^-)\right)=\alpha\cup \delta.
\qedhere\]
\end{proof}
%
%
Note that a bigon move fixes the number of components of $F\cap S^+\cap S^-$---which equals half the number of points of $F\cap\partial(S^+\cap S^-)$---if and only if  $\alpha\cap\partial\nu L=\varnothing$; otherwise, a bigon move increases this number, and thus the complexity of $F$. 
In particular:
\begin{lemma}\label{L:MoveCase}
If sequence of bigon moves begins with the complexity of $F$ minimized, and if all bigon moves in this sequence along arcs disjoint from $\partial\nu L$ precede all other bigon moves in this sequence, then this sequence of bigon moves preserves the fact that all components of $F\setminus(S^+\cup S^-\cup\nu L)$ are disks. 
\end{lemma}

\begin{proof}
Because bigon moves along arcs disjoint from $\partial\nu L$ fix the complexity of $F$, all such moves in this sequence are performed while the complexity of $F$ is still minimized. Satisfying the second sufficient condition from Lemma \ref{L:Move}, these bigon moves preserve the fact that all components of $F\setminus(S^+\cup S^-\cup\nu L)$ are disks.  All remaining bigon moves follow arcs that intersect $\partial\nu L$,  meeting the first sufficient condition from Lemma \ref{L:Move}. Therefore, these moves too preserve the fact that all components of $F\setminus(S^+\cup S^-\cup\nu L)$ are disks.
\end{proof}


\subsection{A key lemma}\label{S:lemma}
 
Several proofs in \textsection\ref{S:3}-\ref{S:4} will use the following lemma. 


%
\begin{lemma}\label{L:inessential}
If $\alpha\subset\partial\nu L\cap S^\pm$ is an arc whose endpoints $\partial\alpha=\alpha\cap F$ lie on distinct circles of $F\cap\partial\nu L$, then these endpoints also lie on distinct circles of $F\cap S^\pm$.  
\end{lemma}
%
%
\begin{proof}
Let $\alpha\subset\partial\nu L\cap S^\pm$ be an arc whose endpoints $\partial\alpha=\alpha\cap F$ lie the same circle $\gamma$ of $F\cap S^\pm$. We claim that these endpoints must also lie on the same circle of $F\cap\partial\nu L$.   
Since $\gamma$ bounds a disk $Y\subset F\cap B^\pm$, there is an arc $\beta\subset Y$ with the same endpoints as $\alpha$. The circle $\alpha\cup\beta$ lies on the boundary of some component of $B^\pm\setminus F$, a ball, and therefore bounds a disk $Z\subset B^\pm$ whose interior is disjoint from $F$. 
 The arc $\alpha\subset\partial\nu L$ is parallel in the link exterior through $Z$ to $F$; $\partial$-incompressibility implies that $\alpha$ must also be parallel in $\partial\nu L$ to $F$, and in particular that the endpoints of $\alpha$ must lie on the same component of $F\cap\partial\nu L$, as claimed.
\end{proof}
The following special case is particularly noteworthy:
\begin{lemma}\label{L:inessentialCase}
The two arcs of $F\cap\partial\nu L$ traversing each over/underpass lie on distinct circles of $F\cap S^\pm$.
\end{lemma}
\begin{proof}
This follows immediately from Lemma \ref{L:inessential}.
\end{proof}
\subsection{Crossing tubes}\label{S:ct}
Say that $F$ has a standard \textbf{tube} near a crossing ball $C_t$ if there are two arcs $\alpha_1,\alpha_2\subset F\cap S^+\cap S^-$ such that (1) for $r=1,2$, there is an isotopy of $(\alpha_r,\partial\alpha_r)$ through $(S^+\cap S^-\setminus F, S^+\cap S^-\cap\partial\nu L)$ to $(\partial C_t,S^+\cap S^-\cap\partial\nu L\cap\partial C_t)$---i.e., for $r=1,2$, $\alpha_r$ is parallel through $S^+\cap S^-$ to $C_t$, allowing the endpoints to slide along $S^+\cap S^-\cap\partial\nu L$---and (2) the endpoints of $\alpha_1$, $\alpha_2$ are also endpoints of the four arcs of $F\cap\partial\nu L\cap S^\pm$ that traverse the overpass and underpass at $C_t$---i.e. these endpoints are (among) the points of $F\cap\partial\nu L\cap S^+\cap S^-$ closest to $C_t$ in each of the four directions along $\partial\nu L$, in the sense of the disk-bundle $\pi:\nu L\to L$.
\begin{figure}
\begin{center}
\includegraphics[height=1.75in]{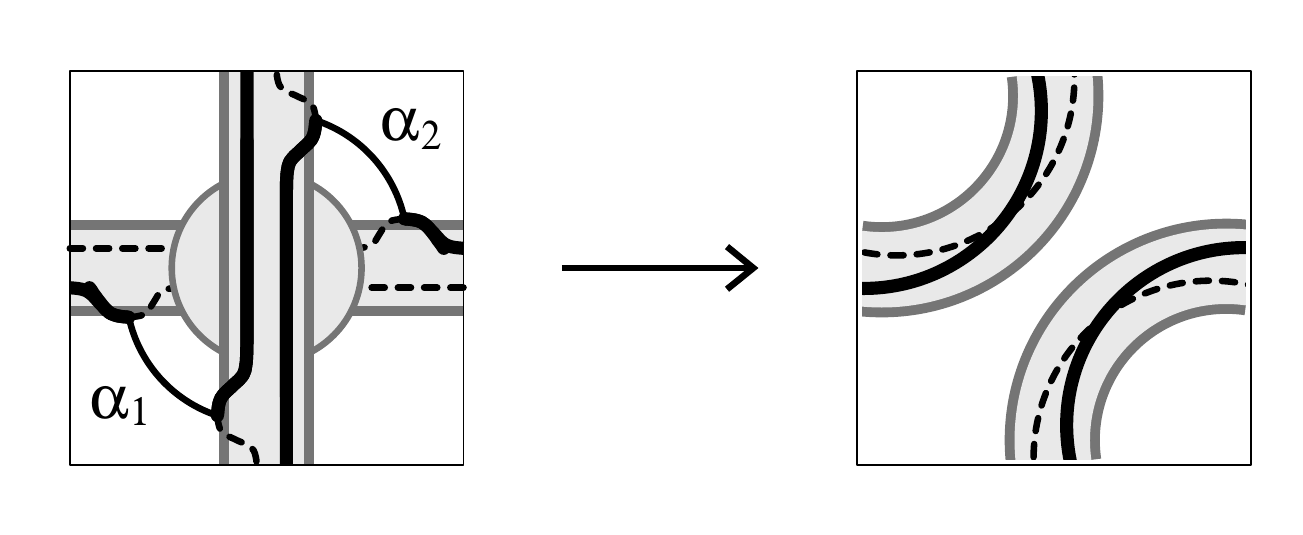}
\hspace{.15in}
\raisebox{.875in}{$\begin{matrix}
\raisebox{-.075in}{\includegraphics[width=.55in]{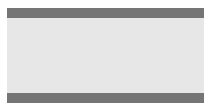}}:&\partial\nu L\\~\\
\includegraphics[width=.5in]{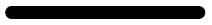}:&F\cap\partial\nu L\cap S^+\\~\\
\includegraphics[width=.5in]{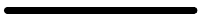}:&F\cap S^+\cap S^-\\~\\
\includegraphics[width=.5in]{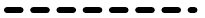}:&F\cap\partial\nu L\cap S^-\\
\end{matrix}$}
\caption{A \textbf{tube} near a crossing ball $C_t$ features two arcs $\alpha_1,\alpha_2\subset F\cap S^+\cap S^-$ parallel through $S^+\cap S^-\setminus F$ to $C_t$ whose endpoints are (among) the points of $F\cap\partial\nu L\cap S^+\cap S^-$ closest to $C_t$ in each direction along $\partial\nu L$. In a minimal crossing link diagram, such a tube gives a compressing disk $Z$ for $F$ in $S^3$ with $|\partial Z\cap L|=2$ (cf. Lemma \ref{L:sufficient}), implying that $r(F,L)\leq 2$. Compressing $F$ along $Z$ changes $\nu L$, $F$ as shown.}
\label{Fi:Smoothing}
\end{center}
\end{figure}
Up to symmetry, this appears as in Figure \ref{Fi:Smoothing}---the arcs $\alpha_1$, $\alpha_2$ must be in opposite quadrants relative to $C_t$, not adjacent ones.  The reason for this is that the only other possibility (cf. Figure \ref{Fi:TwoInnermosts}, second from left) contradicts the essentiality of $F$ in the link exterior, using Lemma \ref{L:inessential}, specifically Lemma \ref{L:inessentialCase}.

A crossing tube sets up a surgery move on $F$, $L$ as follows.  
Each of the two arcs $\alpha_1$, $\alpha_2$ associated with a crossing tube near $C_t$ has one endpoint on an edge-pair incident to the overpass at $C_t$ and the other on an edge-pair incident to the underpass at $C_t$.  The two endpoints on edge-pairs incident to the overpass at $C_t$ can be joined by an arc $\beta_1\subset\partial\nu L\cap (S^-\cup C_t)\setminus F$; likewise, the two endpoints on edge-pairs incident to the underpass at $C_t$ can be joined by an arc $\beta_2\subset\partial\nu L\cap (S^+\cup C_t)\setminus F$. 

The circle $\alpha_1\cup\beta_1\cup\alpha_2\cup\beta_2$ bounds a disk $X$ whose interior lies in $S^3\setminus (F\cup\nu L)$.  Both $\beta_1$, $\beta_2$ are parallel through disks $Y_1,Y_2\subset\nu L$ to arcs $\beta_1',\beta_2'\subset F$ each of which intersects $L$ in a single point.  Thus, the disk $Z:=X\cup Y_1\cup Y_2$ satisfies $Z\cap F=\partial Z=\alpha_1\cup\beta_1'\cup\alpha_2\cup\beta_2'$ with $|\partial Z\cap L|=2$.  Moreover, the arcs $\beta_1'$, $\beta_2'$ can be isotoped so that the two points (total) in which they intersect $L$ are the endpoints of the vertical crossing arc in $C_t$, one on the overpass and the other on the underpass.

The surgery move associated to the crossing tube near $C_t$ (cf. Figure \ref{Fi:Smoothing}) consists of (1) cutting $F$ along $\partial Z$, while cutting $L$ at the two points of $\partial Z\cap L$, and (2) gluing in two parallel copies of the disk $Z$, while joining each pair of endpoints of $L$ on the boundary of the glued-in copy of the disk $Z$ with an arc in that disk.  This surgery move is a compression of $F$ in $S^3$ unless $\partial Z$ bounds a disk in $F$, in which case the surgery move yields a surface with two components, one of them a sphere.  The effect of the surgery move on $L$ is the same as one of two possible ``smoothings'' near $C_t$, in the traditional sense from skein relations. After the surgery move, the resulting link is again embedded in the resulting surface.
\begin{lemma}\label{L:sufficient}
Given a crossing tube in a minimal crossing diagram of a non-split link, the associated surgery move is 
 a compression---the boundary of the surgery disk $Z$ does not bound a disk in $F$.
\end{lemma}
In particular, a crossing tube in a reduced alternating link diagram contains a ``genuine'' compressing disk $Z$ for $F$ in $S^3$ with $|\partial Z\cap L|=2$.  This lemma will round off the proof of the main theorem in \textsection\ref{S:4}.
\begin{proof}
Let $C_t$ be the crossing ball with the tube in question. Construct the disk $Z$ as before, and suppose for contradiction that $\partial Z$ bounds a disk $Y$ in $F$. 
Since $|\partial Y\cap L|=|\partial Z\cap L|=2$ and $L$ is non-split, $Y\cap L$ consists of a single arc, call it $\delta$. From $(Y\cup Z)\cap L=\delta$ it follows that the two 2-spheres on the boundary of a thin regular neighborhood of $Y\cup Z$ intersect $L$ in a total of two points; hence, one of these 2-spheres is disjoint from $L$.  Since $L$ is non-split, this implies that the disks $Y$ and $Z$ are parallel through a ball $W$ in $S^3$ \textit{whose interior is disjoint from $L$}.
 
The arc $\delta=Y\cap L$ is parallel through the ball $W$ to any arc $\delta'$ in $Z$ that joins the overpass and underpass of $L$ at $C_t$. Taking $\delta'$ to be the vertical arc in $C_t$ that joins the overpass and underpass of $L$, isotope the arc $\delta\subset L$ to $\delta'$, while fixing the rest of $L$.   This isotopy eliminates all crossings incident to $\delta$, including the one at $C_t$, without creating any new ones. We assumed this was impossible.
\end{proof}
\subsection{Height}\label{S:height}
In the broader setting of  \textsection\ref{S:broader}, construct graphs $G^\pm$ (they will be trees) whose vertices correspond to the components of $S^\pm$ cut along $F$ and whose edges correspond to the components of $F\cap S^\pm$, such that
the edge corresponding to each circle $\gamma\subset F\cap S^\pm$ joins the vertices corresponding to the two components of $S^\pm$ cut along $F$ whose boundaries contain $\gamma$. 
The two ``sides of $\gamma$ in $S^\pm$'' refer to the two components of $S^\pm$ cut along $\gamma$ and correspond to the two components of the graph $G^\pm_\gamma$ obtained from $G^\pm$ by deleting the interior of the edge associated with $\gamma$.  Define the \textbf{height} of $\gamma$ on each side to be the \textit{maximum} edge-length among all (simple) paths in the corresponding component of $G^\pm_\gamma$ that start at the (appropriate endpoint of the) edge associated with $\gamma$. 

Thus, innermost circles of $F\cap S^\pm$ have height 0 (to one side), circles which enclose (to one side) innermost circles and no others have height 1, and so on.  Figure \ref{Fi:T34} shows an example---a torus containing the knot $T_{3,4}$ in an almost-alternating diagram, with the associated graph $G^+$. 
\begin{prop}\label{P:Height1Exists}
Either all circles of $F\cap S^\pm$ have height 0 (to one side), or $F\cap S^\pm$ contains a circle with height 1 (to one side). 
\end{prop}
\begin{proof}
Let $\gamma$ be a circle of $F\cap S^\pm$ with height $k\geq 2$ (to one side)---if no such $\gamma$ exists, there is nothing to prove. Consider a maximal edge-length path in $G^\pm_{\gamma}$ that starts at the appropriate end of the edge associated with $\gamma$. This path traverses $k$ edges, 
the $r^ \text{th}$ of which corresponds to a circle of $F\cap S^\pm$ with height $k-r$.
\end{proof}

\section{Consequences of minimality and essentiality}\label{S:3}
To review the setup and preliminary results, a link $L\subset (S^2\setminus C)\cup\partial C$ follows a 
link diagram $D\subset S^2$ with $n$ crossings; $\nu L$ is a regular neighborhood of $L$, seen as (the total space) of a disk-bundle $\pi:\nu L\to L$;  balls $B^\pm$ are the closures of the two components of $S^3\setminus (S^2\cup C\cup\nu L)$, with $S^\pm=\partial B^\pm$ and $S^+\cap S^-=S^2\setminus \text{int}(C\cup\nu L)$; and a closed surface $F$, which contains $L$ and is \textbf{essential} (i.e. incompressible and $\partial$-incompressible in the link exterior, i.e. $r(F,L)\geq 2$), has been isotoped such that:
\begin{itemize}  
\item The restriction $\pi|_F$ is a bundle map;
\item $F$ is transverse to $S^+$, $S^-$;  
\item $F\cap C\cap\partial\nu L=\varnothing$;
\end{itemize}
 and, subject to these conditions, the \textbf{complexity} of $F$ is \textbf{minimized}, specifically:
\begin{itemize}
\item The numbers of components of $F\cap\partial C\setminus\nu L$ and of $F\cap S^+\cap S^-$ have lexicographically been minimized. 
\end{itemize}
With this initial setup, \textsection\ref{S:24} showed that all components of $F\cap B^+$, $F\cap B^-$ and $F\cap C$ are disks; all components of $F\cap\partial\nu L\cap S^\pm$, $F\cap\partial C\cap S^\pm$, and $F\cap S^+\cap S^-$ are arcs; and all components of $B^+\setminus F$ and $B^-\setminus F$ are balls.  Actually, the last two statements follow from the first one, using the essentiality of $F$.  In the somewhat more general setting where this first condition replaces the assumption of minimal complexity, \textsection\ref{S:bigon}-\ref{S:height} established such technical conveniences as bigon moves, crossing tubes, and height.  

Section \ref{S:3} delimits which local configurations are consistent with the initial setup, where $F$ is essential and its complexity is minimized. Most results address the arcs of $F\cap\partial\nu L\cap S^\pm$, $F\cap\partial C\cap S^\pm$, and $F\cap S^+\cap S^-$, and many extend to the case where $F$ is incompressible but $\partial$-compressible in the link exterior.  (One way to extend the proofs, roughly, is to isotope $F$ so as to push any $\partial$-compressing disks into $\nu L$, and then to slide these disks along the link away from the local area under consideration.) 

Several of the proofs require disrupting minimality by creating new components of $F\cap S^+\cap S^-$, usually through a sequence of bigon moves. There are several valid reasons to do this.  In the proof of Lemma \ref{L:EdgeCrossingLoop}, a \textit{temporary} increase in complexity enables the removal of a component of $F\cap\partial C\setminus\nu L$, lessening the complexity of $F$, a contradiction. One case in the proof of Lemma \ref{L:TwoInnermosts} increases complexity in order to reveal a $\partial$-compressing disk, thus contradicting essentiality. The other case in the proof of Lemma \ref{L:TwoInnermosts}---and the proof of the crossing tube lemma in \textsection\ref{S:4}---increase complexity in order to procure a crossing tube. In such cases, Lemma \ref{L:Move} will confirm that the bigon moves, while disrupting minimality, preserve at least the fact that all components of $F\cap B^+$, $F\cap B^-$ and $F\cap C$ are disks, and thus the more general setting of \textsection\ref{S:bigon}-\ref{S:height}. In particular, this will validate further bigon moves and applications of Lemma \ref{L:inessential}. 

%
\subsection{Local possibilities, regardless of alternatingness}
\label{S:31}%
Assume throughout \textsection\ref{S:31} that $D$ is a diagram of a non-trivial, non-split link $L$, and that $L$ is contained in a closed surface $F\subset S^3$ (compact and connected without boundary). Establish all setup from \textsection\ref{S:21}-\ref{S:23}. Assume in particular that $F$ is essential and its complexity $\left(|F\cap\partial C\setminus\nu L|,|F\cap S^+\cap S^-|\right)$ has been minimized.
%
\begin{prop}\label{P:EdgeLoopB}
No arc of $F\cap\partial\nu L\cap S^\pm $ has both endpoints on the same component of $\partial\nu L\cap S^+\cap S^-$. 
\end{prop}
\begin{prop}\label{P:CrossingLoopAA}
No arc of $F\cap\partial C\cap S^\pm$ has both endpoints on the same component of $\partial C\cap S^+\cap S^-$.
\end{prop}
\begin{figure}
\begin{center}
\includegraphics[height=1.75in]{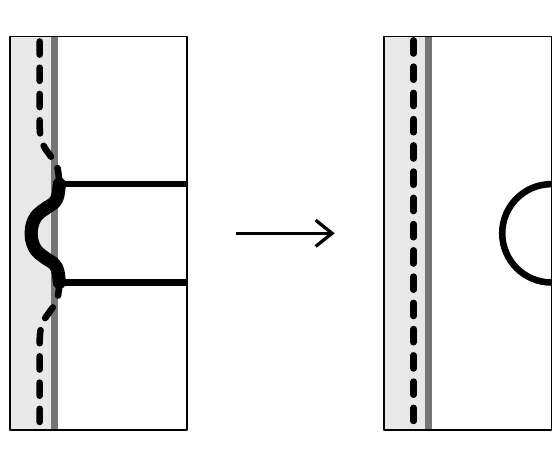}
\hspace{.5in}
\includegraphics[height=1.75in]{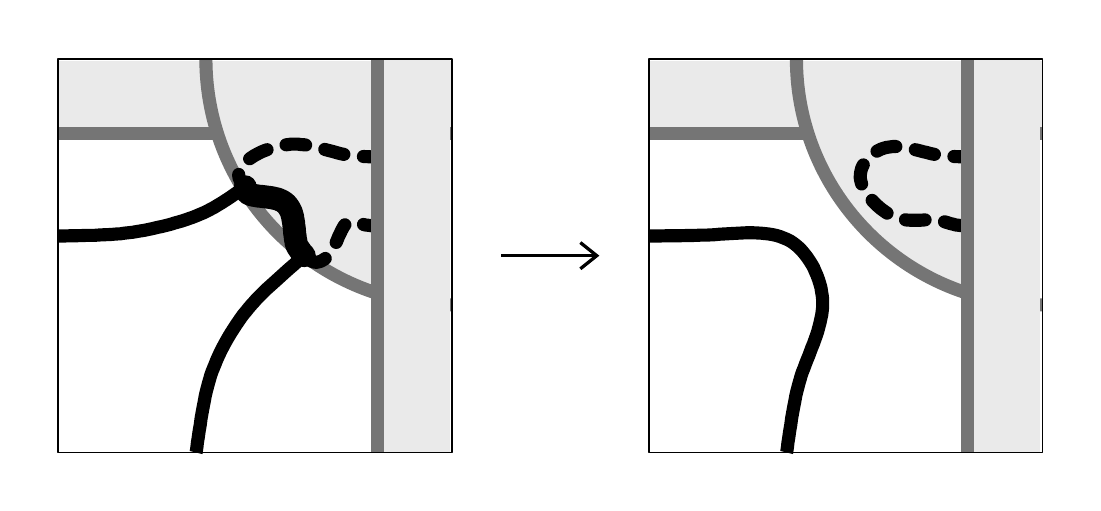}
\caption{No arc of  $F\cap\partial\nu L\cap S^\pm$ is parallel in $\partial\nu L$  to   $S^+\cap S^-$ (left), and no arc of $F\cap\partial C\cap S^\pm$ is parallel in  $\partial C$  to   $S^+\cap S^-$ (right).}
\label{Fi:EdgeLoop}
\end{center}
\end{figure}
 See Figure \ref{Fi:EdgeLoop}. Recall that $F\cap\partial\nu L\cap  C=\varnothing$ by assumption.
\begin{proof}[Proof of Propositions \ref{P:EdgeLoopB}, \ref{P:CrossingLoopAA}]
Any arc $\alpha_1$ of $F\cap\partial\nu L\cap S^\pm$ with endpoints on the same component of $\partial\nu L\cap S^+\cap S^-$ must be parallel through a disk $X\subset\partial\nu L\cap S^\pm$ to an arc $\beta_1\subset\partial\nu L\cap S^+\cap S^-$; but then, isotoping $F$ near $\alpha_1$ through $X$ past $\beta_1$ (Figure \ref{Fi:EdgeLoop}, left) would reduce the complexity of $F$, contrary to assumption. 

Likewise, any arc $\alpha_2$ of $F\cap\partial C\cap S^\pm$ with endpoints on the same component of $\partial C\cap S^+\cap S^-$ must be parallel through a disk $Y\subset\partial C\cap S^\pm$ to an arc $\beta_2\subset\partial C\cap S^+\cap S^-$; but then, isotoping $F$ near $\alpha_2$ through $Y$ past $\beta_2$ (Figure \ref{Fi:EdgeLoop}, right) would reduce the complexity of $F$, contrary to assumption. 
\end{proof}
\begin{lemma}\label{L:saddle}
Every component of $F\cap C\setminus\nu L$ is a disk whose boundary consists of four arcs, alternately on $S^+\cap\partial C$ and $S^-\cap\partial C$, none of which is parallel in $\partial C\setminus\nu L$ to $\partial C\cap S^+\cap S^-$. 
\end{lemma}
\begin{figure}
\begin{center}
\includegraphics[height=1.75in]{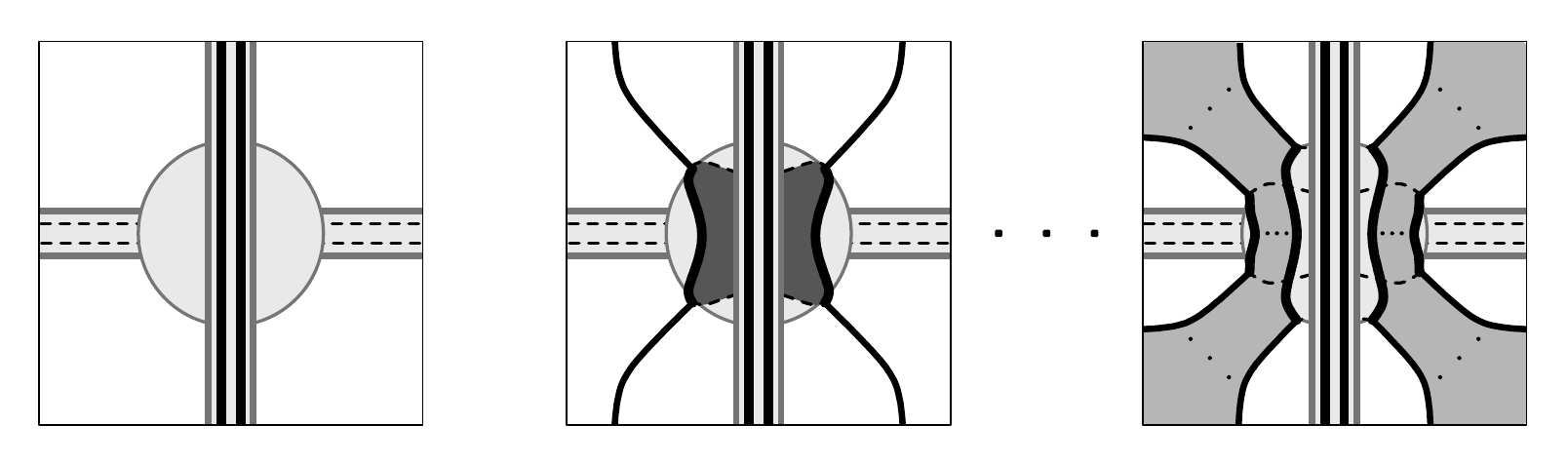}
\caption{An arbitrary crossing ball: (left) disjoint from $F$, (center) intersecting $F$ in a single component, (right) intersecting $F$ in at least two components 
(cf. Lemma \ref{L:saddle}). }
\label{Fi:CrossingSaddle}
\end{center}
\end{figure}
 That is, each component of $F\cap C\setminus\nu L$ looks like a saddle, as in Figure \ref{Fi:CrossingSaddle} (center), and each crossing ball looks like one of the pictures in Figure \ref{Fi:CrossingSaddle}, depending on the number of components in which it intersects $F$.
\begin{proof}
 Lemma \ref{L:saddle} is an immediate consequence of Propositions \ref{P:Disks}, \ref{P:Arcs}, and \ref{P:CrossingLoopAA}.
\end{proof}
\begin{prop}\label{P:EdgeLoopE}
No arc of $F\cap S^+\cap S^-$ has both endpoints on the same component of $\partial\nu L\cap S^+\cap S^-$.
\end{prop}
\begin{prop}\label{P:CrossingLoopE}
No arc of $F\cap S^+\cap S^-$ has both endpoints on the same crossing ball.  
\end{prop} 
\begin{figure}
\begin{center}
\includegraphics[height=1.75in]{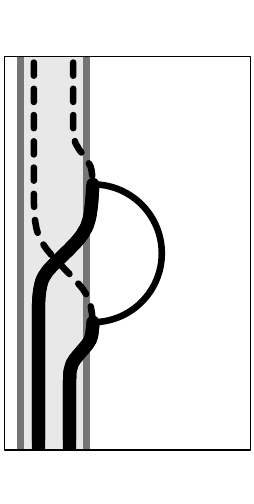}
\hspace{.5in}
\includegraphics[height=1.75in]{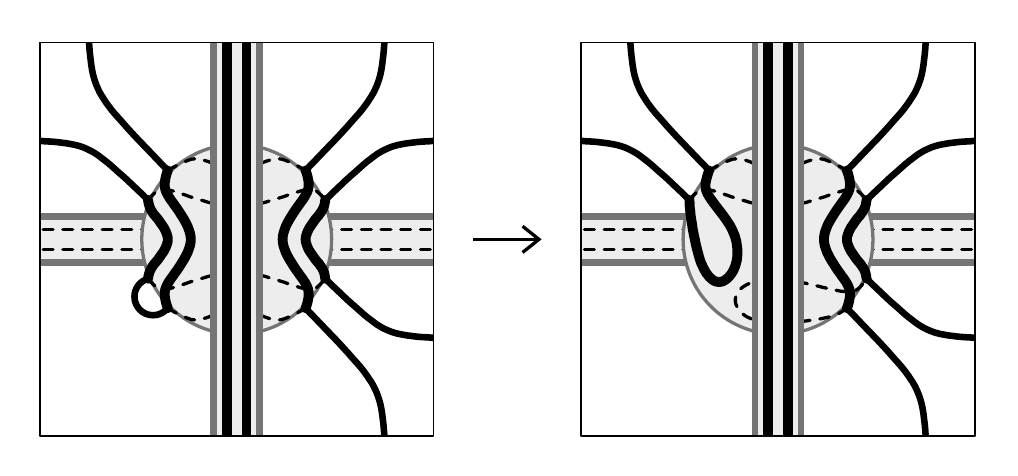}
\caption{No arc of $F\cap S^+\cap S^-$ is parallel in $S^+\cap S^-$ to  $\partial\nu L$ (left) or to $\partial C$ (right).}
\label{Fi:CrossingLoop}
\end{center}
\end{figure}
\begin{proof}[Proof of Propositions \ref{P:EdgeLoopE}, \ref{P:CrossingLoopE}]
If both endpoints of some arc of $F\cap S^+\cap S^-$ were on the same component of $\partial\nu L\cap S^+\cap S^-$ then, applying Proposition \ref{P:EdgeLoopB}, an outermost such arc in $S^+\cap S^-$  would appear as left in Figure \ref{Fi:CrossingLoop}, contradicting the assumed $\partial$-incompressibility of $F$ in the link exterior, e.g. by Lemma \ref{L:inessential}.

Suppose instead that $\alpha_0$ is an arc of $F\cap S^+\cap S^-$ with both endpoints on the same crossing ball $C_t$, and assume that $\alpha_0$ is outermost in $S^+\cap S^-$, i.e. parallel through $S^+\cap S^-\setminus F$ to $C_t$, as in Figure \ref{Fi:CrossingLoop}.  Push $F$ near $\alpha_0$ through $S^+\cap S^-$ past $\partial C_t$. This attaches two saddle-shaped disks in the interior of $C_t$, lessening the complexity of $F$, contrary to assumption.
\end{proof}
\begin{lemma}\label{L:EdgeCrossingLoop}
No arc $\alpha_0\subset F\cap S^+\cap S^-$ has 
one endpoint on a crossing ball and the other on an incident edge of $\partial\nu L$.
\end{lemma}
\begin{proof}
Suppose that $\alpha_0$ is an arc of $F\cap S^+\cap S^-$ with one endpoint on a crossing ball $C_t$ and the other on an incident edge of $\partial\nu L$. 
Assume that $\alpha_0$ is outermost in $S^+\cap S^-$, i.e. parallel through $S^+\cap S^-\setminus F$ to $\nu L\cup C$. Consider the circle of $F\cap\partial\nu L$ that contains an endpoint of $\alpha_0$.  Moving along this circle from that endpoint toward $C_t$, there is at most one more point on $S^+\cap S^-$, by Proposition \ref{P:EdgeLoopB} and the assumption that $\alpha_0$ is outermost.  There are thus two cases up to symmetry (cf. Figure \ref{Fi:EdgeCrossingLoop}).

\begin{figure}
\begin{center}
\includegraphics[width=6.5in]{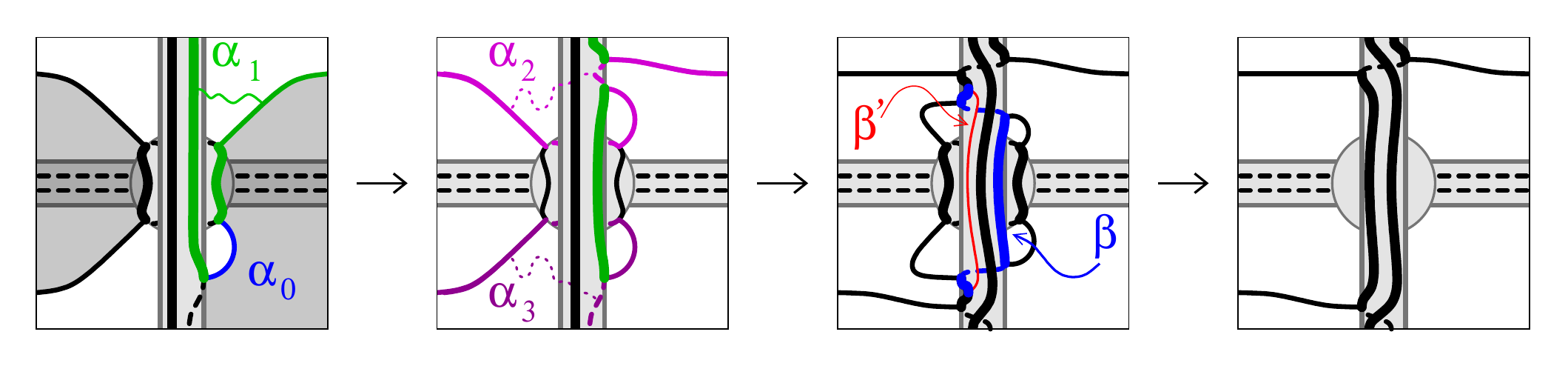}
\includegraphics[width=6.5in]{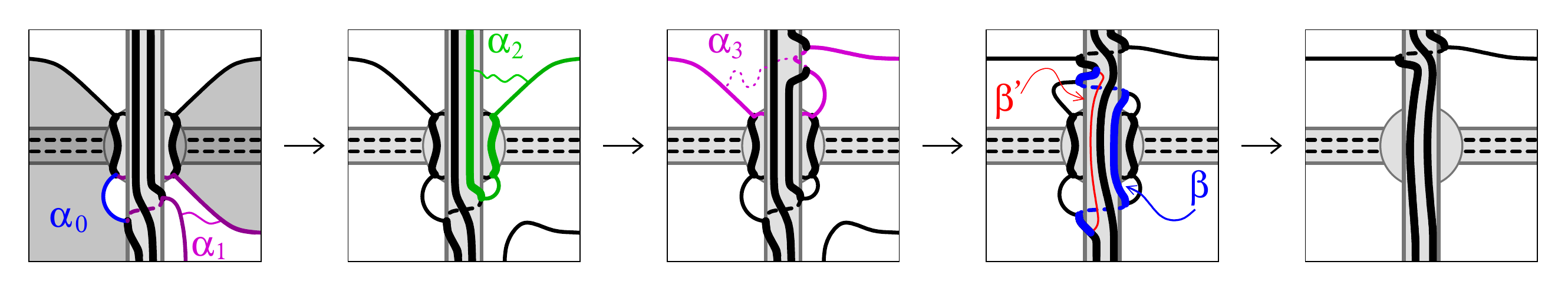}
\caption{No arc $\alpha_0$ of $F\cap S^+\cap S^-$ has one endpoint on a crossing ball and the other on an incident edge of $\partial\nu L$.}
\label{Fi:EdgeCrossingLoop} 
\end{center}
\end{figure}
In either case, begin with a sequence of (up to) three bigon moves,
through the arcs labeled $\alpha_1$, $\alpha_2$, $\alpha_3$ in Figure \ref{Fi:EdgeCrossingLoop}, in that order.  Any of the arcs $\alpha_r$ can be parallel in $S^\pm\setminus C$ to $F$; in this case, omit the bigon move along $\alpha_r$. 
In all cases, this sequence of bigon moves fits the hypotheses of Lemma \ref{L:MoveCase} and thus preserves the fact that all components of $F\setminus(S^+\cup S^-\cup\nu L)$ are disks.  

Now an arc
 $\beta\subset\partial\nu L\cap F$ is parallel through a disk $Y\subset\partial\nu L$ with $Y\cap F=\beta$ to a second arc $\beta'\subset\partial\nu L\cap S^\pm$ with $\beta'\cap F=\partial\beta'$. 
This arc $\beta'$ is parallel to an arc $\beta''\subset F\cap B^\pm$ through a disk $Z\subset B^\pm$ with $Z\cap F=\beta''$. 
Further, $\beta''$ is parallel to $\beta$ through a disk $X\subset F$, which contains an entire disk of $F\cap C\setminus\nu L$.
Finally, the 2-sphere $X\cup Y\cup Z$ bounds a ball $W$ in the link exterior.

Isotope $(X,\beta)$ through $(W,Y)$ to $(Z,\beta')$, while fixing $\beta''=\partial X\cap\partial Z$.  This removes the disk of $X\cap C$ and thus a component of $F\cap C\cap\partial\setminus\nu L$.  Since bigon moves always fix $|F\cap\partial C\setminus\nu L|$, this contradicts the initial assumption that the complexity of $F$ was minimized.
\end{proof}
\subsection{Consequences of alternatingness.}
\label{S:32}
Maintain all setup from \textsection\ref{S:31}, with the additional assumption that $D$ is alternating.  That is, assume throughout \textsection\ref{S:32} that $D$ is a non-trivial, connected, alternating diagram of a link $L$, and that $L$ is contained in a closed essential surface $F\subset S^3$ (compact and connected without boundary) whose complexity $\left(|F\cap\partial C\setminus\nu L|,|F\cap S^+\cap S^-|\right)$ has been minimized.
\begin{figure}
\begin{center}
\includegraphics[width=6.5in]{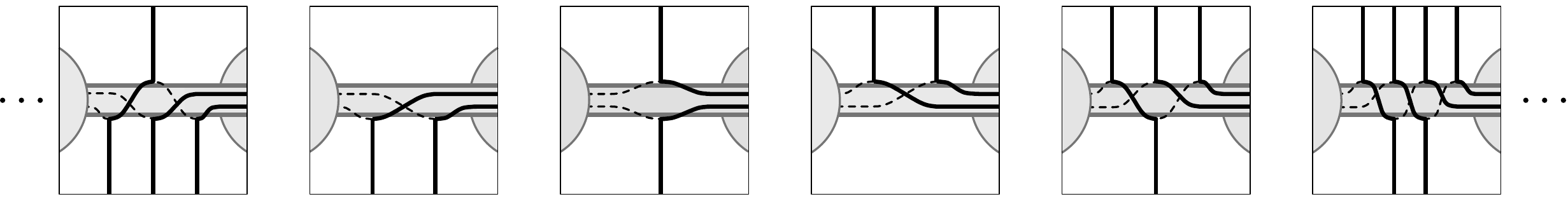}
\caption{The types of edge-pairs of $\partial\nu L$ when $D$ is alternating (cf. Lemma \ref{L:EdgeTypes}).}
\label{Fi:EdgeTypes}
\end{center}
\end{figure} 
\begin{lemma}\label{L:EdgeTypes}
 Every edge-pair of $\partial\nu L$ appears as in Figure \ref{Fi:EdgeTypes}.
\end{lemma}
\begin{proof}
 This follows immediately from Propositions \ref{P:Arcs} and \ref{P:EdgeLoopB} and the alternatingness of $D$.
\end{proof}
\newpage
\begin{prop}\label{P:InnermostCrossingBall}
If $\gamma$ is an innermost circle of $F\cap S^\pm$, then $\gamma\cap\partial C=\varnothing$.
\end{prop}
\begin{figure}
\begin{center}
\includegraphics[width=6.5in]{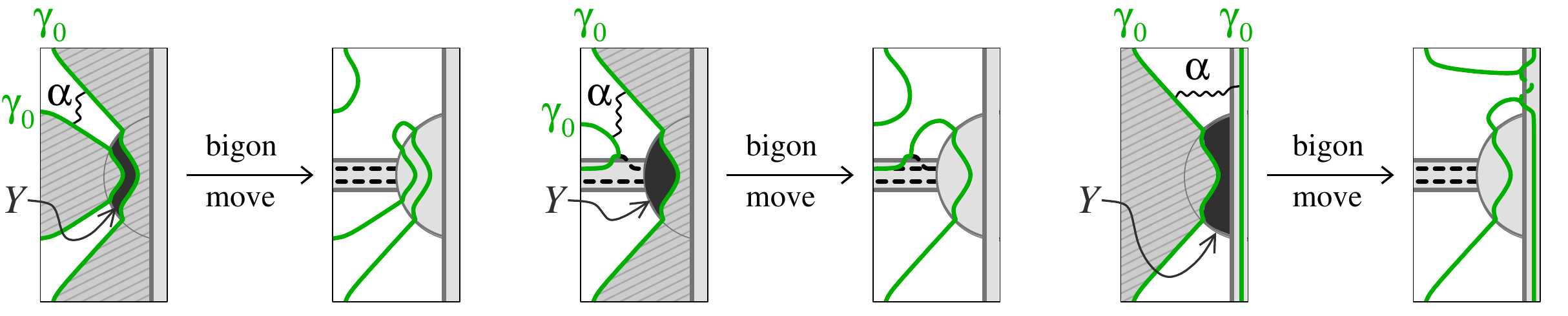}
\caption{Every innermost circle $\gamma_0\subset F\cap S^\pm$ is disjoint from $\partial C$ (cf. Proposition \ref{P:InnermostCrossingBall}).}\label{Fi:InnermostCrossingBall}
\end{center}
\end{figure} 
\begin{proof}
Let $\gamma_0$ be an innermost circle of $F\cap S^+$ (\textsc{wlog}), and suppose $\gamma_0\cap\partial C\neq\varnothing$.  Then the disk $X\subset S^+\setminus F$ with $\partial X=\gamma_0$ intersects $\partial C$; let $Y$ be a component of $X\cap\partial C$.  There are now three cases, using Lemma \ref{L:saddle} (cf. Figure \ref{Fi:InnermostCrossingBall}): in one case (left), $\partial Y$ contains two arcs of $\gamma_0\cap\partial C$. Otherwise  $\partial Y$ contains an arc on $\partial\nu L$, along either the underpass (center) or the overpass (right) at $C_t$.

If $\partial Y$ contains more than one arc of $\gamma_0\cap C_t$ (left in Figure \ref{Fi:InnermostCrossingBall}), then a bigon move yields an arc of $F\cap S^+\cap S^-$ with both endpoints on $C_t$.  This contradicts the minimality of $F\cap\partial C\setminus\nu L$ (recall Figure \ref{Fi:CrossingLoop}, right, and the proof of Proposition \ref{P:CrossingLoopE}). 
%

In both remaining cases (center and right), a bigon move yields an arc of $F\cap S^+\cap S^-$ with one endpoint on $C_t$ and the other on an incident edge of $\partial\nu L$. Crucial to the initial bigon moves in Figure \ref{Fi:InnermostCrossingBall}, center and right, is that every edge of $\partial\nu L$ contains endpoints of $F\cap S^+\cap S^-$, by alternatingness. Next, perform the sequence of bigon moves from the top of Figure \ref{Fi:EdgeCrossingLoop} (omitting any trivial ones as usual), to set up the final isotopy move from that sequence.  This final move is valid, since the preceding sequence of bigon moves meets the conditions of Lemma \ref{L:MoveCase}, and it removes a component of $F\cap C\cap\partial\setminus\nu L$.  Since bigon moves always fix $|F\cap\partial C\setminus\nu L|$, this contradicts the initial assumption that the complexity of $F$ was minimized.
\end{proof}
\begin{prop}\label{P:InnermostPass}
If $\gamma$ is an innermost circle of $F\cap S^\pm$, so that $\gamma$ bounds a disk $X\subset S^\pm$, then at least one component of $\gamma\cap\partial\nu  L$ traverses an over/underpass.
\end{prop}
\begin{proof}
Lemma \ref{L:EdgeTypes} 
implies that the only arcs of $F\cap\partial\nu L$ with endpoints on the same component of $\nu L\cap S^2\setminus C$ look like the arcs of this type in Figure \ref{Fi:EdgeTypes}, up to reflection---there is one far left in Figure \ref{Fi:EdgeTypes}, one second from right, and two far right.  If such an arc lies on an innermost circle $\gamma$ of $F\cap S^+$ (\textsc{wlog}), then, since $\gamma$ is innermost and $\gamma\cap\partial C=\varnothing$ by Proposition \ref{P:InnermostCrossingBall},  $\gamma$ must traverse the overpass at the crossing where the edge of $\partial\nu L$ containing this arc of $\gamma$ meets an underpass. This is evident in Figure \ref{Fi:EdgeTypes}, using Lemma \ref{L:inessential}.
\end{proof}
\begin{lemma}\label{L:TwoInnermosts}
If both circles of $F\cap S^\pm$ traversing a given over/underpass have height 0 (to one side), i.e. are innermost, then $F$ can be isotoped to have a standard tube near that crossing.
\end{lemma}
\begin{figure}
\begin{center}
\includegraphics[width=3in]{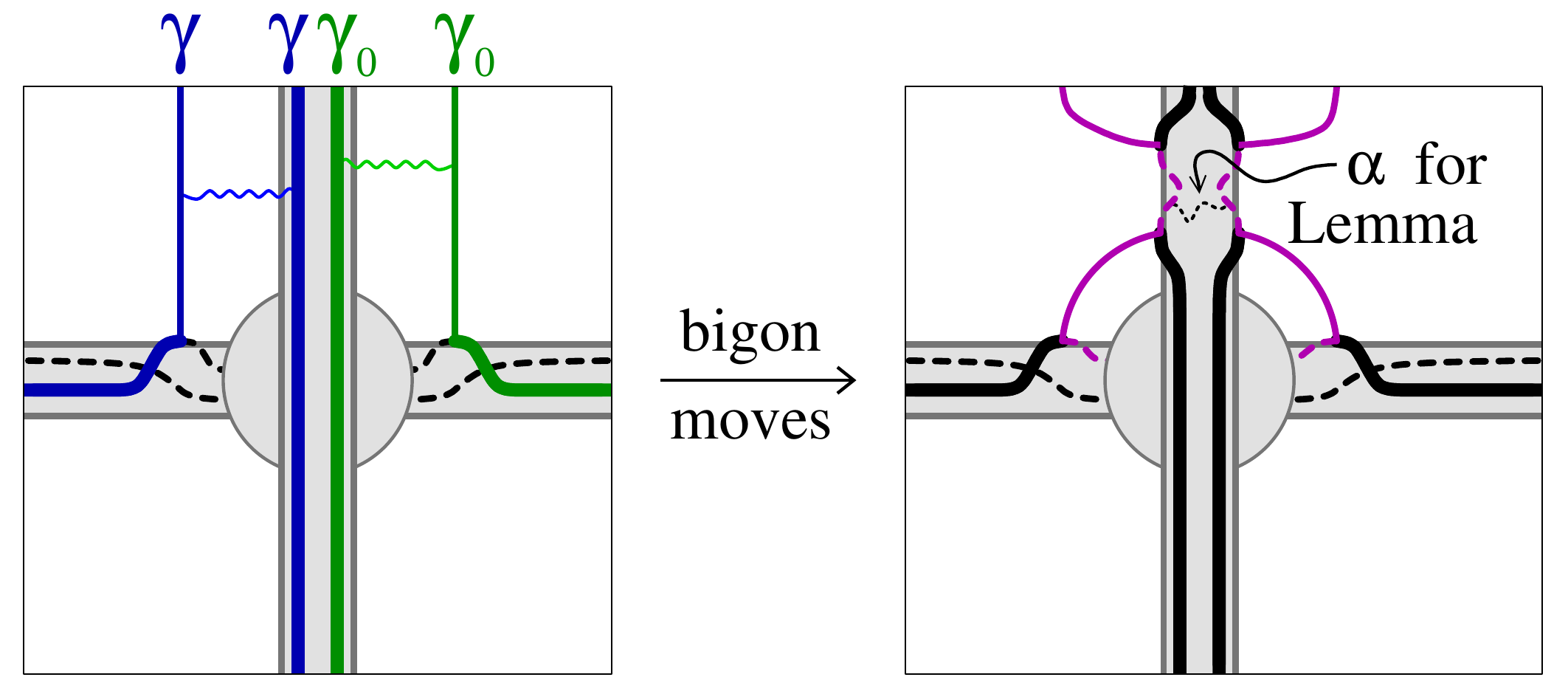}\hspace{.45in}
\includegraphics[width=3in]{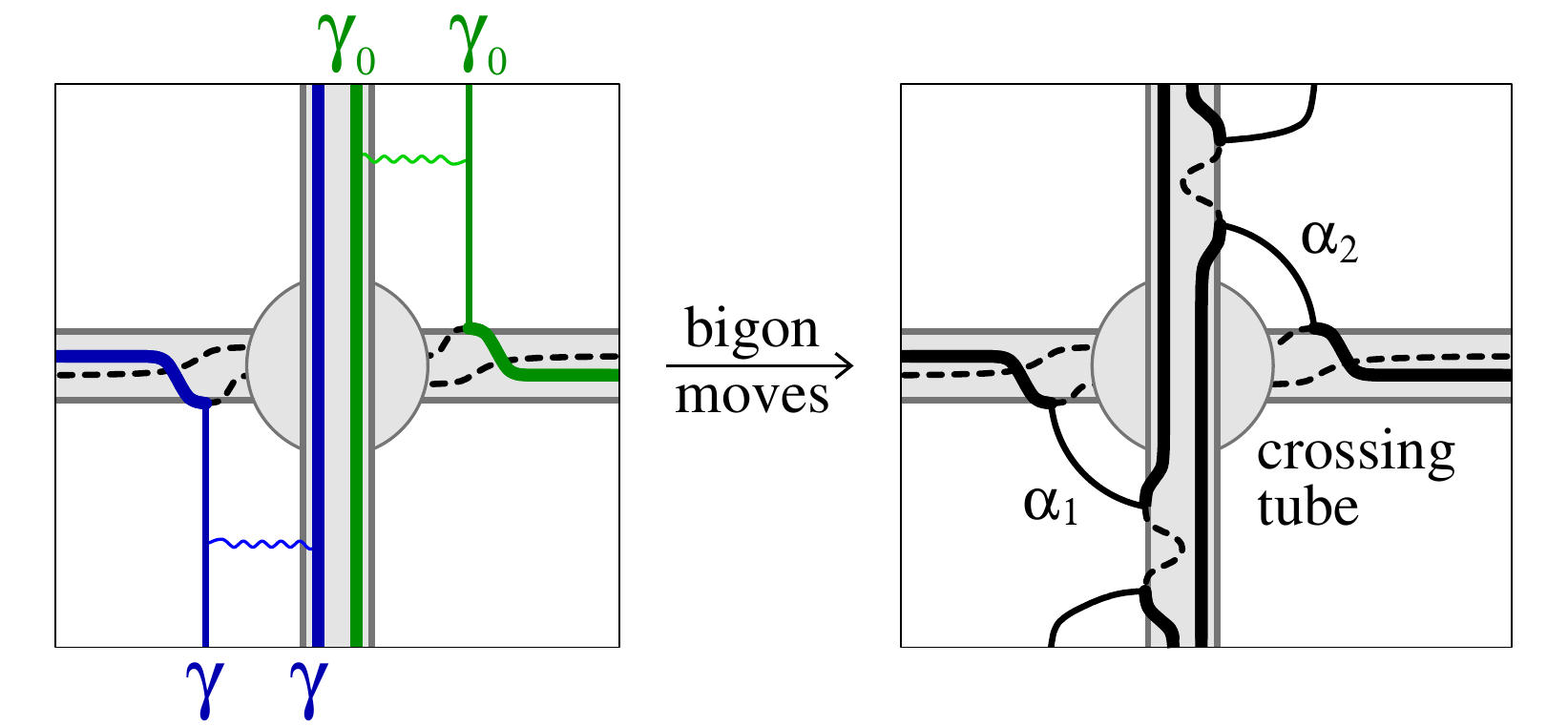}
\caption{If both circles of $F\cap S^\pm$ traversing a given over/underpass, say at $C_t$, have height zero, then $F$ can be isotoped to have a standard tube near $C_t$.}
\label{Fi:TwoInnermosts}
\end{center}
\end{figure}
\begin{proof}
Let $C_t$ be the ball at the crossing in question. Recall that the two arcs traversing the overpass (\textsc{wlog}) at $C_t$ lie on distinct circles of $F\cap S^+$, by Lemma \ref{L:inessentialCase}. Consider the two edge-pairs of $\partial\nu L$ that are incident to the underpass at $C_t$.  Because $D$ is alternating, each edge-pair contains endpoints of $F\cap S^+\cap S^-$. Consider the endpoint on each edge-pair that is nearest to $C_t$ in $\partial\nu L$---recall that $F\cap\partial\nu L$ is (the image of) a section of the disk-bundle $\nu L\to L$.  Up to symmetry, there are two cases, depending on which sides of $D$ these two points lie on, relative to each other (cf. Figure \ref{Fi:TwoInnermosts}). 

If these two points lie in adjacent quadrants near $C_t$ (left in Figure \ref{Fi:TwoInnermosts}), perform two bigon moves (unless the associated arc is parallel in $S^\pm\setminus C$ to $F$). Since each (possible) bigon move follows an arc with an endpoint on $\partial\nu L$, Lemma \ref{L:Move} implies that these moves preserve the fact that $S^+\cup S^-\cup \nu L$ cuts $F$ into disks.  Moreover, they produce a diagram in which an arc $\alpha\subset\partial\nu L\cap S^-$ has endpoints on the same circle of $F\cap S^-$ but on distinct circles of $F\cap\partial\nu L$, which Lemma \ref{L:inessential} states is impossible.

Therefore, these two points must lie in opposite quadrants $F\setminus D$ near $C_t$ (right in Figure \ref{Fi:TwoInnermosts}).  In this case, a pair of bigon moves (to be omitted if trivial) immediately fashions a standard tube near $C_t$.  Lemma \ref{L:MoveCase} implies that these moves preserve the fact that $S^+\cup S^-\cup \nu L$ cuts $F$ into disks.
\end{proof}
\section{Main results}\label{S:4}
\begin{ctlemma}
Given a non-trivial, connected, reduced alternating diagram of a link $L$ and a closed, essential surface $F\supset L$, 
there exists an isotopy after which $F$ has a standard {tube} near some crossing.
\end{ctlemma}
\begin{proof}
As in \textsection\ref{S:2}, let $L\subset (S^2\setminus C)\cup\partial C$ follow a reduced alternating diagram, with $\nu L$ a closed regular neighborhood of $L$ seen as (the total space of) a disk-bundle $\pi:\nu L\to L$, $B^\pm$ the components of $S^3\setminus (S^2\cup \text{int}(C\cup\nu L))$, and $S^\pm=\partial B^\pm$. Let $F$ be a closed, essential surface containing $L$.  Fixing $L\subset F$, $S^2$, and $C$, isotope $F$---subject to the requirements that $F\pitchfork S^+,S^-$; $\pi|_F$ is a bundle map; and $F\cap C\cap\partial\nu L=\varnothing$---so as to minimize lexicographically the numbers of components of $F\cap\partial C\setminus\nu L$ and $F\cap S^+\cap S^-$. 

Consider $F\cap S^+$.  If all circles have height 0, apply Lemma \ref{L:TwoInnermosts} at any overpass, done.   Otherwise, by Proposition \ref{P:Height1Exists}, there exists a circle $\gamma_1$ of $F\cap S^+$ with height 1.  Let $\gamma_0$ be any (innermost) circle enclosed by $\gamma_1$.  Apply Proposition \ref{P:InnermostPass} to consider an overpass which $\gamma_0$ traverses. Let $\gamma$ denote the other circle of $F\cap S^+$ traversing this overpass.  Note that $\gamma\neq\gamma_0$ by Lemma \ref{L:inessential}, specifically Lemma \ref{L:inessentialCase}.  If $\gamma$ has height 0, then Lemma \ref{L:TwoInnermosts} completes the proof.  Otherwise, $\gamma$ must equal $\gamma_1$.
See Figure \ref{Fi:FinalMove}.

\begin{figure}
\begin{center}
\includegraphics[width=6.5in]{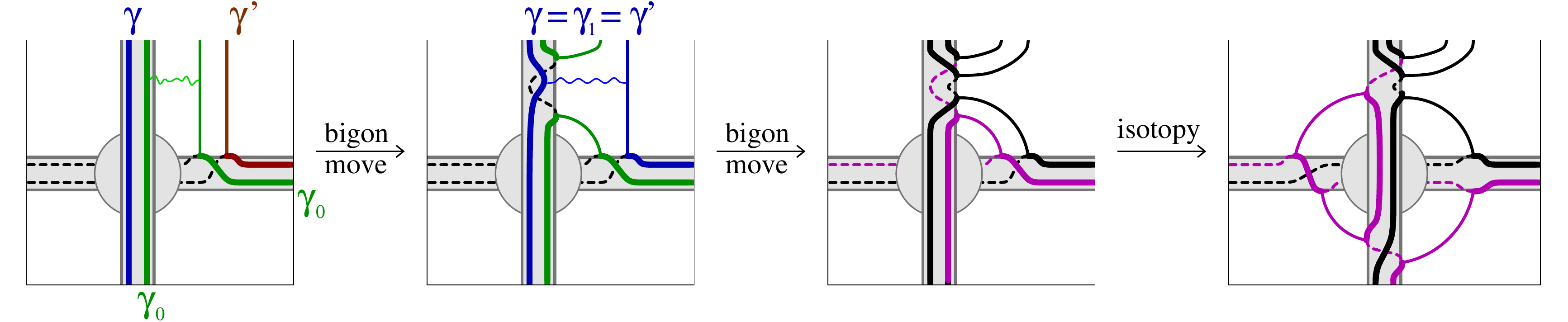}
\caption{The final sequence of moves in the proof of the crossing tube lemma.}
\label{Fi:FinalMove}
\end{center}
\end{figure}
Next, consider the circle $\gamma'$ of $F\cap S^+$ from Figure \ref{Fi:FinalMove}, which must exist and be distinct from $\gamma_0$, due to Lemmas \ref{L:EdgeTypes}, \ref{L:inessential}, and the assumption that $\gamma_0$ has height 0. 
If $\gamma'$ also has height 0, then Lemma \ref{L:inessential} 
 implies that the edge of $\partial\nu L$ in question must appear as in Figure \ref{Fi:EdgeTypes}, second from right; thus, $\gamma_0$ and $\gamma'$ must traverse a common overpass, completing the proof, using Lemma \ref{L:TwoInnermosts}.  
Otherwise, $\gamma'=\gamma_1=\gamma$.  This allows the sequence of isotopy moves shown in Figure \ref{Fi:FinalMove}, yielding the desired crossing tube. (Again, omit either isotopy move if the associated arc is parallel in $S^\pm\setminus C$ to $F$; and Lemma \ref{L:MoveCase} applies since each bigon move is along an arc with an endpoint on $\partial\nu L$.)
\end{proof}
%
%
\begin{maintheorem}
Every non-split, non-trivial alternating link $L$ has representativity $r(L)=2$.
\end{maintheorem}
\begin{proof}
Beginning with a {reduced} alternating diagram of $L$, apply the crossing tube lemma to obtain a standard tube at some crossing.  Then apply Lemma \ref{L:sufficient} to conclude that the crossing tube contains a disk $Z$ with $Z\cap F=\partial Z$, such that $\partial Z$ intersects $L$ in two points and does not bound a disk in $F$.  
\end{proof}
\section{Conclusion}\label{S:5}
\begin{figure}
\begin{center}
\includegraphics[height=3.5in]{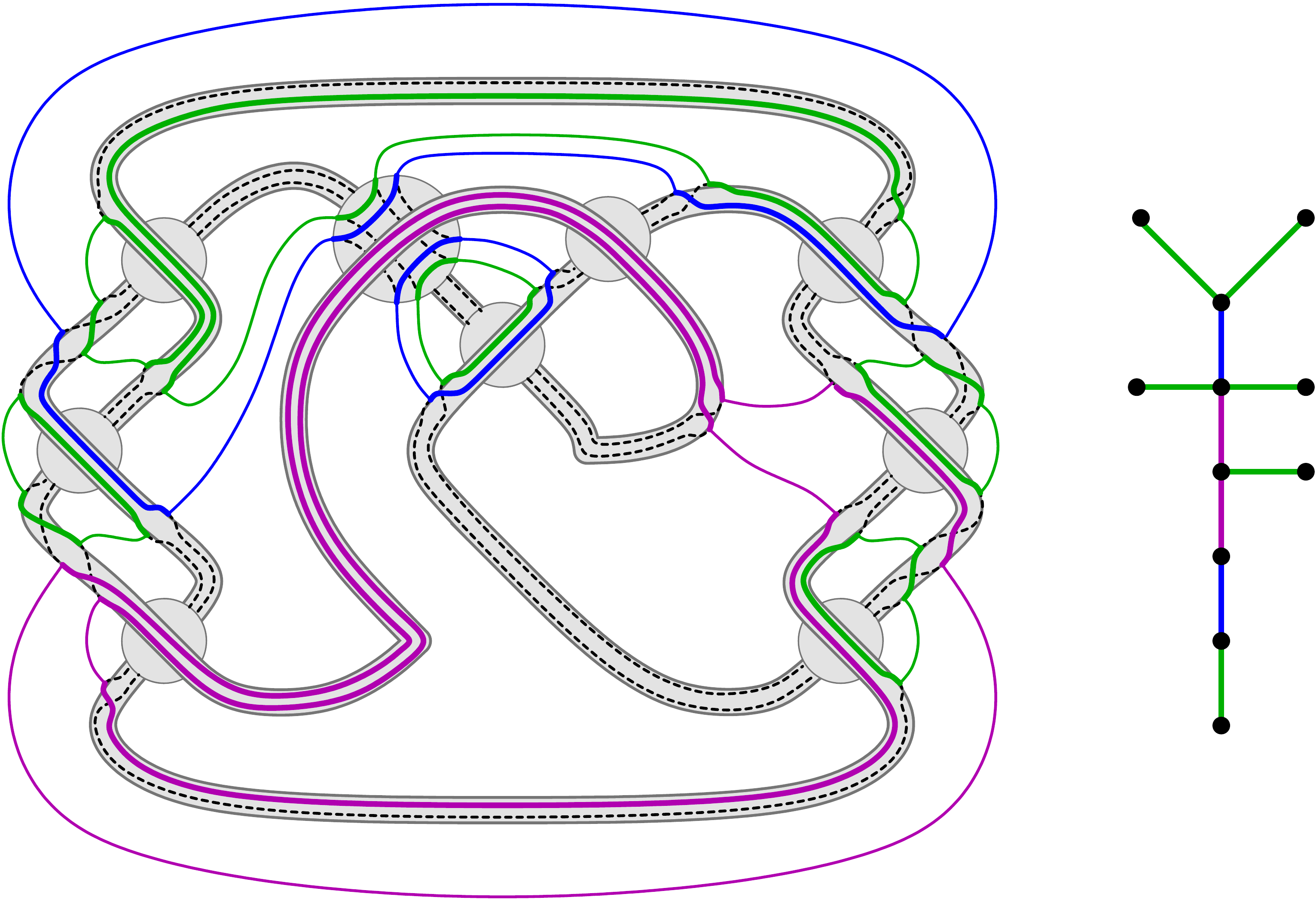}
\caption{A torus $F$ containing the knot $T_{3,4}$ in an almost-alternating diagram, with the 
graph $G^+$.  There are two circles of $F\cap S^+$ with height 2, two with height 1, and six 
 with height 0.}\label{Fi:T34}
\end{center}
\end{figure}
Figure \ref{Fi:T34} shows a torus $F$ containing the knot $L=T_{3,4}$, with $L$ in an almost-alternating diagram. The torus is difficult to visualize directly from this diagram, but its homeomorphism type is straightforward to verify by computing euler characteristic. As $r(L)\geq r(F,L)=3>2$,
this example illustrates how the arguments leading up to the crossing tube lemma break down without alternatingness. The configuration satisfies the conclusions of all Propositions and Lemmas from \textsection\ref{S:31}, but escapes those from \textsection\ref{S:32} (regarding edge types and innermost possibilities). 

If an almost-alternating link $L$ has representativity $r(L)\geq 3$, then there can be no crossing tube in any diagram with minimal crossing number. Thus, after applying the setup and results from \textsection2, \textsection3.1, all non-innermost circles of $F\cap S^+$ and $F\cap S^-$ must (properly enclose or) be incident to the ``de-alternator.'' Otherwise, a local application of the proof of the crossing tube lemma gives $r(F)\leq 2$.  Likewise, at least one of the two circles traversing each over/underpass must be incident to (or properly enclose) the de-alternator.  Perhaps it is feasible to list explicitly which almost alternating links have representativity $\geq 3$.



%
\end{document}